\documentclass[12pt,reqno]{amsart}
\usepackage{amssymb}
\usepackage{amsxtra}
\usepackage{mathrsfs}
\usepackage[all]{xy}
\usepackage{cite}
\usepackage{paralist}
\usepackage[hypertex]{hyperref}
\numberwithin{equation}{section}
\newtheorem{theorem}{Theorem}[section]
\newtheorem{lemma}[theorem]{Lemma}
\newtheorem{prop}[theorem]{Proposition}
\newtheorem{corollary}[theorem]{Corollary}
\theoremstyle{definition}
\newtheorem{definition}[theorem]{Definition}
\theoremstyle{remark}
\newtheorem{example}[theorem]{Example}
\newtheorem{remark}[theorem]{Remark}
\newtheorem*{ackn}{Acknowledgments}
\DeclareMathOperator{\Hom}{Hom}
\DeclareMathOperator{\Ker}{Ker}
\DeclareMathOperator{\Der}{Der}
\newcommand*{\alg}{\mathrm{alg}}
\newcommand*{\smth}{\mathrm{smth}}
\newcommand*{\hol}{\mathrm{hol}}
\newcommand*{\env}{\mathrm{e}}
\newcommand*{\dL}{\mathrm{L}}
\newcommand*{\lmod}{\mbox{-}\!\mathop{\mathsf{mod}}}
\newcommand*{\rmod}{\mathop{\mathsf{mod}}\!\mbox{-}}
\newcommand*{\bimod}{\mbox{-}\!\mathop{\mathsf{mod}}\!\mbox{-}}
\newcommand*{\lalg}{\mbox{-}\!\mathop{\mathsf{alg}}}
\newcommand*{\Ptens}{\mathop{\widehat\otimes}}
\newcommand*{\Ptensalg}{\Ptens\lalg}
\newcommand*{\Tens}{\mathop{\otimes}}
\newcommand*{\ptens}[1]{\mathop{\widehat\otimes}_{#1}}
\newcommand*{\Lptens}[1]{\mathop{\widehat\otimes}_{#1}\nolimits^{\dL}}
\renewcommand*{\dh}{\mathop{\mathrm{dh}}}
\newcommand*{\db}{\mathop{\mathrm{db}}}
\newcommand*{\id}{1}
\newcommand*{\CC}{\mathbb C}
\newcommand*{\N}{\mathbb N}
\newcommand*{\Z}{\mathbb Z}

\newcommand*{\DD}{\mathbb D}
\newcommand*{\TT}{\mathbb T}
\newcommand*{\cO}{\mathscr O}
\newcommand*{\cB}{\mathscr B}
\newcommand*{\cT}{\mathcal T}
\newcommand*{\cK}{\mathscr K}
\newcommand*{\bD}{\mathsf D}
\newcommand*{\h}{\mathbf h}
\newcommand*{\spn}{\mathrm{span}}
\newcommand*{\eps}{\varepsilon}
\newcommand*{\ol}{\overline}
\newenvironment{mycompactenum}{\pltopsep=5pt\begin{compactenum}[\upshape (i)]}%
{\end{compactenum}}
\newcommand*{\xra}{\xrightarrow}
\newcommand{\lriso}{\stackrel{\textstyle\sim}{\smash\longrightarrow%
\vphantom{\scriptscriptstyle{_1}}}}
\begin{document}
\title[Dense quasi-free subalgebras of the Toeplitz algebra]{Dense quasi-free
subalgebras\\ of the Toeplitz algebra}
\author{A. Yu. Pirkovskii}
\address{Alexei Yu. Pirkovskii, Faculty of Mathematics,
HSE University,
6 Usacheva, 119048 Moscow, Russia}
\email{aupirkovskii@hse.ru}
\thanks{This work was supported by the RFBR grant no. 19-01-00447.}
\date{}

\begin{abstract}
We introduce a family of dense subalgebras of the Toeplitz algebra and give
conditions under which our algebras are quasi-free. As a corollary, we show that the smooth
Toeplitz algebra introduced by Cuntz is quasi-free.
\end{abstract}

\maketitle

\section{Introduction}

Quasi-free algebras were introduced by Schelter \cite{Schelter} under the name of
``smooth algebras''. The main idea behind this notion is that, for a number of reasons,
quasi-free algebras can be viewed as noncommutative analogs of smooth affine varieties
(or, more exactly, as analogs of algebras of functions on smooth affine varieties
in the category of all associative algebras). This point of view was further developed
by Cuntz and Quillen \cite{CQ1}, who coined the name ``quasi-free'' for this class of algebras,
mostly because they behave like free algebras with respect to nilpotent extensions.
Cuntz and Quillen gave several useful characterizations of quasi-free algebras,
provided many examples, and proved a number of interesting properties of such algebras.
Quasi-free algebras and their generalizations were also considered in the functional
analytic context, both in the locally convex \cite{Cuntz_cycl,Pir_qfree}
and in the bornological settings
\cite{Meyer_thesis,Meyer_exentire,Meyer_locan,Voigt_per,Voigt_loc}.
They play an important role in cyclic homology theory
\cite{CQ2,CQ_ex,Cuntz_cycl,Meyer_thesis,Meyer_exentire,Meyer_locan,
Quillen_book,Voigt_per,Voigt_loc}
and in some other aspects of noncommutative geometry \cite{LeBruyn_book,Ginzburg,BZ}.

In this paper, we study quasi-freeness for some dense subalgebras of the
Toeplitz algebra, i.e., the universal $C^*$-algebra $\cT$ generated by an isometry.
Here we understand quasi-freeness in the setting of locally convex algebras,
but we believe that the bornological approach is also possible.
The starting point for our study is the fact (probably due to Meyer \cite{Meyer_locan})
that the algebraic Toeplitz algebra, i.e., the $*$-subalgebra $\cT_\alg$ of $\cT$
algebraically generated by the ``universal'' isometry $v\in\cT$, is quasi-free
(see Proposition~\ref{prop:Talg_qfree} below).
On the other hand, $\cT$ itself is not quasi-free. This is a simple corollary of a very general
result due to Aristov \cite{Ar}. In retrospective, this is not surprising at all, because
there are numerous results (going back to Helemskii's global dimension theorem \cite{X_dg_ne_1})
showing that, for Banach algebras, the property of being quasi-free is a rare phenomenon.

We will be interested in locally convex algebras continuously embedded in $\cT$
and containing $\cT_\alg$. The best known example of such an algebra is the smooth
Toeplitz algebra introduced by Cuntz \cite{Cuntz_biv_lok}.
Another natural example is the holomorphic Toeplitz algebra recently considered
by Panarin \cite{Panarin}. Our main goal is to construct a family $\{ \cT_{P,Q}\}$
of locally convex subalgebras of $\cT$ (where $P$ and $Q$ are K\"othe sets
satisfying some natural conditions, see Section~\ref{sect:TPQ} for details)
and to give a sufficient condition
for the inclusion $\cT_\alg\hookrightarrow\cT_{P,Q}$ to be a homological epimorphism. This will
imply the quasi-freeness of $\cT_{P,Q}$. As a corollary, we show that the smooth Toeplitz algebra
and the holomorphic Toeplitz algebra are quasi-free.

\begin{ackn}
The author thanks Oleg Aristov for the question of whether or not the smooth Toeplitz algebra
is quasi-free. This question was the main motivation for writing this paper.
\end{ackn}

\section{Preliminaries}
\subsection{Locally convex algebras and modules}
This subsection gives a brief account of homological algebra
in categories of locally convex modules.
Our main reference is \cite{X1}; some details can also be found in
\cite{X2,X_HOA,T1,EschmPut,Pir_qfree}.

Throughout, all vector spaces and algebras are assumed to be over the field $\CC$
of complex numbers. All algebras are assumed to be associative and unital.
By a {\em $\Ptens$-algebra} we mean an algebra $A$ endowed with
a complete locally convex topology in such a way that the multiplication
$A\times A\to A$ is jointly continuous.
Note that the multiplication uniquely extends to a continuous linear map
${A\Ptens A\to A},\; a\otimes b\mapsto ab$,
where the symbol $\Ptens$ stands for the completed projective
tensor product (whence the name ``$\Ptens$-algebra'').
If the topology on $A$ can be determined by a family of submultiplicative
seminorms (i.e., a family $\{\|\cdot\|_\lambda : \lambda\in\Lambda\}$ of
seminorms such that $\| ab\|_\lambda\le\| a\|_\lambda \| b\|_\lambda$
for all $a,b\in A$), then $A$ is said to be {\em locally $m$-convex}
(or an {\em Arens-Michael algebra}).
A {\em Fr\'echet algebra} is a $\Ptens$-algebra $A$ whose underlying locally convex
space is metrizable.
The category of all $\Ptens$-algebras and continuous algebra homomorphisms
will be denoted by $\Ptensalg$.

Let $A$ be a $\Ptens$-algebra.
A {\em left $A$-$\Ptens$-module} is a left $A$-module $M$ endowed with
a complete locally convex topology in such a way that
the action $A\times M\to M$ is jointly continuous.
We always assume that $1_A\cdot x=x$ for all $x\in M$, where $1_A$ is the identity of $A$.
Left $A$-$\Ptens$-modules and their continuous $A$-module morphisms form a category
denoted by $A\lmod$.
The categories $\rmod A$ and $A\bimod A$ of right
$A$-$\Ptens$-modules and of $A$-$\Ptens$-bimodules are defined similarly.
Note that $A\bimod A\cong A^\env\lmod\cong\rmod A^\env$,
where $A^\env=A\Ptens A^{\mathrm{op}}$,
and where $A^{\mathrm{op}}$ stands for the algebra opposite to $A$.
Given $M,N\in A\lmod$ (respectively, $\rmod A$, $A\bimod A$),
the space of morphisms from $M$ to $N$ will be denoted by ${_A}\h(M,N)$
(respectively, $\h_A(M,N)$, ${_A}\h_A(M,N)$).

If $M$ is a right $A$-$\Ptens$-module and $N$
is a left $A$-$\Ptens$-module, then their {\em $A$-module tensor product}
$M\ptens{A}N$ is defined to be
the completion of the quotient $(M\Ptens N)/L$, where $L\subset M\Ptens N$
is the closed linear span of all elements of the form
$x\cdot a\otimes y-x\otimes a\cdot y$
($x\in M$, $y\in N$, $a\in A$).
As in pure algebra, the $A$-module tensor product can be characterized
by the universal property that, for each complete locally convex space $E$,
there is a natural bijection between the set of all jointly
continuous $A$-balanced bilinear maps from $M\times N$ to $E$
and the set of all continuous linear maps from
$M\ptens{A}N$ to $E$.

A chain complex $C=(C_n,d_n)_{n\in\Z}$ in $A\lmod$ is {\em admissible} if
it is split exact in the category of topological vector spaces, i.e., if it has
a contracting homotopy consisting of continuous linear maps. Geometrically,
this means that each $d_n$ is an open map of $C_n$ onto $\Ker d_{n-1}$
(which implies, in particular, that $C$ is exact in the purely algebraic sense),
and that $\Ker d_n$ is a complemented subspace
of $C_n$ for each $n$.
The category $A\lmod$ together with the class of all short admissible sequences
is an exact category in the sense of Quillen \cite{Quillen}.
This implies that the derived categories $\bD(A\lmod)$, $\bD^\pm(A\lmod)$, and
$\bD^b(A\lmod)$ are defined (see \cite{Keller_DCU,Pir_qfree} for details).
The same is true of $\rmod A$ and $A\bimod A$.

A left $A$-$\Ptens$-module $P$ is {\em projective}
if the functor ${_A}\h(P,-)$ takes admissible sequences
of $A$-$\Ptens$-modules to exact sequences of vector spaces.
A {\em projective resolution} of $M\in A\lmod$ is a pair $(P,\eps)$,
where $P$ is a nonnegative chain complex consisting of projective $A$-$\Ptens$-modules
and $\eps$ is a morphism from $P_0$ to $M$ such that the sequence
$P\xra{\eps} M\to 0$ is an admissible complex.
The {\em length} of $P$ is the minimum integer $n$
such that $P_i=0$ for all $i>n$, or $\infty$ if there is no such $n$.
It is a standard fact that $A\lmod$ has {\em enough projectives},
i.e., each left $A$-$\Ptens$-module has a projective resolution.
The same is true of $\rmod A$ and $A\bimod A$.

The {\em projective homological dimension} of $M\in A\lmod$ is the
minimum integer $n=\dh_A M\in \Z_+\cup\{\infty\}$
(where $\Z_+$ is the set of all nonnegative integers)
with the property that $M$ has a projective resolution of length $n$.
The {\em bidimension} of $A$ is defined by
$\db A=\dh_{A^\env} A$.

Given a $\Ptens$-algebra $A$ and an $A$-$\Ptens$-bimodule $M$,
we let $\Der(A,M)$ denote the space of all continuous derivations of
$A$ with values in $M$. We will need the following standard
characterization of derivations. Let us equip $A\times M$ with a
$\Ptens$-algebra structure by letting $(a,m)(b,n)=(ab,an+mb)$
($a,b\in A$, $m,n\in M$). Let $p_1\colon A\times M\to A$ denote the
projection given by $(a,m)\mapsto a$. We then have a natural
isomorphism
\begin{equation}
\label{der=hom}
\Der(A,M)\cong \{ \varphi\in\Hom_{\Ptensalg}(A,A\times M) : p_1\varphi=\id_A\}.
\end{equation}
Explicitly, the above isomorphism takes each derivation $D\colon A\to M$ to the
homomorphism $A\to A\times M$ given by $a\mapsto (a,D(a))$ (see, e.g., \cite{Pir_qfree}).

If $A$ is a $\Ptens$-algebra, then the bimodule of {\em noncommutative differential
$1$-forms over $A$} is an $A$-$\Ptens$-bimodule $\Omega^1 A$ together with a
derivation $d_A\colon A\to\Omega^1 A$ such that for each $A$-$\Ptens$-bimodule $M$
and each derivation $D\colon A\to M$ there exists a unique $A$-$\Ptens$-bimodule
morphism $\Omega^1 A\to M$ making the following diagram commute:
\[
\xymatrix{
\Omega^1 A \ar[r] & M\\
A \ar[u]^{d_A} \ar[ur]_D
}
\]
In other words, we have a natural isomorphism
\[
{_A}\h_A(\Omega^1 A,M)\cong\Der(A,M) \qquad (M\in A\bimod A).
\]
It is a standard fact (see, e.g., \cite{CQ1,Pir_qfree}) that $\Omega^1 A$ exists
and is isomorphic to the kernel of the multiplication map
$\mu_A\colon A\Ptens A\to A$. Under the above identification, the universal
derivation $d_A\colon A\to \Omega^1 A$ acts by the rule
$d_A(a)=1\otimes a-a\otimes 1$ ($a\in A$). Thus we have an exact sequence
\begin{equation}
\label{Omega^1_can}
\xymatrix{
0 \ar[r] & \Omega^1 A \ar[r]^{j_A} & A\Ptens A \ar[r]^(.6){\mu_A}
& A \ar[r] & 0
}
\end{equation}
in $A\bimod A$, where $j_A$ is uniquely determined by $j_A(d_A(a))=1\otimes a-a\otimes 1$
($a\in A$). Note that \eqref{Omega^1_can} splits in $A\lmod$ and in $\rmod A$
(\cite{Pir_qfree}, cf. also \cite{CQ1}). In particular, \eqref{Omega^1_can} is admissible.

Let $B$ be a $\Ptens$-algebra. By an {\em extension} of $B$ we mean a continuous
open homomorphism $\sigma\colon A\to B$, where $A$ is a $\Ptens$-algebra.
It is convenient to interpret $(B,\sigma)$ as an exact sequence
\begin{equation}
\label{extension}
0 \to I \xra{i} A \xra{\sigma} B \to 0,
\end{equation}
where $I=\Ker\sigma$ and $i$ is the inclusion map. We say that \eqref{extension}
{\em splits} (respectively, that \eqref{extension} is {\em admissible})
if there exists a $\Ptens$-algebra homomorphism (respectively, a continuous linear map)
$j\colon B\to A$ such that $\sigma j=\id_B$. We say that \eqref{extension}
is a {\em square-zero} extension if $I^2=0$.

Following \cite{CQ1} (see also \cite{Cuntz_cycl,Quillen_book,Weibel,Witherspn}),
we say that a $\Ptens$-algebra $A$ is {\em quasi-free} if it satisfies
any (hence all) of the following equivalent conditions:

\begin{enumerate}
\item
Each admissible square-zero extension of $A$ splits.
\item
For each admissible square-zero extension $0\to I\to B\to C\to 0$ of $\Ptens$-algebras
and for each $\Ptens$-algebra homomorphism $A\to C$ there exists a $\Ptens$-algebra
homomorphism $A\to B$ making the following diagram commute:
\[
\xymatrix{
&&& A \ar@{-->}[dl] \ar[d]\\
0 \ar[r] & I \ar[r] & B \ar[r] & C \ar[r] & 0
}
\]
\item
$\Omega^1 A$ is projective in $A\bimod A$.
\item
$\db A\le 1$.
\end{enumerate}

The above list of equivalent conditions can be extended; see, e.g.,
\cite{CQ1,Cuntz_cycl,Quillen_book,Meyer_thesis}.

\begin{remark}
\label{rem:qfree_strongest}
Of course, the notion of a quasi-free algebra makes sense in the purely algebraic
case as well (i.e., in the case of algebras not equipped with a topology),
and this is exactly the case where they were introduced for the first time \cite{Schelter,CQ1}.
In this respect, let us note that each algebra $A$ of at most countable dimension
becomes a $\Ptens$-algebra under the strongest locally convex topology \cite[A.2.8]{BFGP},
and that $A$ is quasi-free
as a $\Ptens$-algebra if and only if $A$ is quasi-free in the purely algebraic sense.
This readily follows, for example, from \cite[Corollary 8.5]{Pir_wdgnucl}.
\end{remark}

By a {\em $\Ptens$-algebra epimorphism} we mean an epimorphism in the category
of all $\Ptens$-algebras, i.e., a $\Ptens$-algebra homomorphism $f\colon A\to B$
such that, whenever $C$ is a $\Ptens$-algebra and $g,h\colon B\to C$ are
$\Ptens$-algebra homomorphisms satisfying $gf=hf$, we have $g=h$.
Equivalently, $f$ is an epimorphism if and only if the canonical map
$B\ptens{A} B\to B$ induced by the multiplication on $B$ is an isomorphism
in $B\bimod B$ (the proof of this fact given in \cite[XI.1]{Stnstrm} for the category of
rings holds verbatim for $\Ptens$-algebras).
Following \cite{GL}, we say that $f$ is a {\em homological epimorphism}
if the canonical morphism $B\Lptens{A} B\to B$ induced by the
multiplication on $B$ is an isomorphism in $\bD^-(B\bimod B)$
(where $\Lptens{A}$ is the total left derived functor of $\ptens{A}$).
Explicitly, this means that for
some (or, equivalently, for each) projective resolution $P\to A$ of $A$ in $A\bimod A$
the complex $B\ptens{A} P\ptens{A} B\to B$ is admissible.
Homological epimorphisms were introduced by J.\,L.\,Taylor \cite{T2}
under the name of {\em absolute localizations}.
Since then, they were rediscovered several times under different names
(see \cite{Dicks,GL,NR,Meyer_emb,BBK}), both in the purely algebraic and in the functional analytic
contexts.

\begin{remark}
\label{rem:homepi_db}
For future reference, observe that, if $f\colon A\to B$ is a homological epimorphism, then we
obviously have $\db B\le\db A$. In particular, if $A$ is quasi-free, then so is $B$.
\end{remark}

\subsection{A survey of Toeplitz algebras}
Recall that the {\em Toeplitz algebra} is the universal
unital $C^*$-algebra $\cT$ generated by an isometry. This means that there is an
isometry $v\in\cT$ with the property that for each unital $C^*$-algebra $A$ and each
isometry $w\in A$ there exists a unique unital $*$-homomorphism
from $\cT$ to $A$ which takes $v$ to $w$. More concretely, $\cT$ can be defined
as the $C^*$-subalgebra of $\cB(\ell^2)$ generated by the right shift operator $v$.
The equivalence of the above definitions follows from Coburn's theorem
(see, e.g., \cite{Murphy} for details). Since $v^* v=1$, we have
\[
\cT=\ol{\spn}\{ v^k (v^*)^\ell : k,\ell\in\Z_+\}\subset\cB(\ell^2).
\]
The Toeplitz algebra can also be characterized in terms of Toeplitz operators as follows.
Given a continuous function $f$ on the circle $\TT=\{ z\in\CC : |z|=1\}$, let
$T_f$ denote the corresponding Toeplitz operator on the Hardy space $H^2=H^2(\TT)$.
We then have
\[
\cT=\{ T_f + K : f\in C(\TT),\; K\in\cK(H^2) \}.
\]
More exactly, the map
\begin{equation}
\label{CK=T}
C(\TT)\oplus\cK(H^2)\to\cT,\quad (f,K)\mapsto T_f+K,
\end{equation}
is a vector space isomorphism (see, e.g., \cite{Davidson,Murphy}).

The {\em algebraic Toeplitz algebra} \cite{Cuntz_cycl,CMR,Meyer_locan}
is the unital $*$-subalgebra $\cT_\alg$ of $\cT$
generated (as a $*$-algebra) by $v$. It can also be interpreted in terms
of Toeplitz operators as follows. Let $M_\infty$ denote the algebra of infinite complex
matrices $a=(a_{ij})_{i,j\in\Z_+}$ such that $a_{ij}=0$ for all but finitely many $i,j\in\Z_+$.
We identify $M_\infty$ with a subalgebra of $\cK(H^2)$ by associating to each $a\in M_\infty$
the operator on $H^2$ whose matrix w.r.t. the trigonometric basis $\{ z^k : k\in\Z_+\}$
is $a$. Then
\[
\cT_\alg=\{ T_f + K : f\in \CC[z,z^{-1}],\; K\in M_\infty \},
\]
where the algebra $\CC[z,z^{-1}]$ of Laurent polynomials is interpreted as the algebra
of trigonometric polynomials on $\TT$.

Recall also \cite{CMR} that $\cT_{\alg}$ is isomorphic to the unital
algebra generated by two elements $u,v$ with relation $uv=1$.
It is well known and easy to show that the elements $v^i u^j$ ($i,j\in\Z_+$)
form a basis of $\cT_\alg$.
Since the dimension of $\cT_\alg$ is countable, we may and will consider $\cT_\alg$
as a $\Ptens$-algebra with respect to the strongest locally convex topology
(cf. Remark~\ref{rem:qfree_strongest}).

As was mentioned in the Introduction, our main objects are
locally convex algebras sitting in between $\cT_\alg$ and $\cT$.
The best known example of such an algebra is the {\em smooth Toeplitz algebra} $\cT_\smth$
introduced by Cuntz \cite{Cuntz_biv_lok} (see also \cite{Cuntz_cycl,CMR,Meyer_locan,Khalkhali}).
To define $\cT_\smth$, consider the algebra
\[
\cK_\smth=\Bigl\{ a=(a_{ij})_{i,j\in\Z_+} : a_{ij}\in\CC,\;
\| a\|_n=\sum_{i,j} |a_{ij}|(1+i+j)^n<\infty\;\forall n\in\Z_+\Bigr\}
\]
of {\em smooth compact operators} introduced by Phillips \cite{Phillips_K_Fre}.
Recall that $\cK_\smth$ is an algebra under the usual matrix multiplication and is a Fr\'echet
algebra for the topology generated by the norms $\|\cdot\|_n$ ($n\in\Z_+$).
Now $\cT_\smth\subset\cT$ is defined as follows:
\[
\cT_\smth=\{ T_f + K : f\in C^\infty(\TT),\; K\in\cK_\smth \}.
\]
The restriction of \eqref{CK=T} to $C^\infty(\TT)\oplus \cK_\smth$
is a vector space isomorphism between $C^\infty(\TT)\oplus \cK_\smth$ and $\cT_\smth$.
This enables us to topologize $\cT_\smth$ so that it becomes a Fr\'echet space.
Moreover \cite{Cuntz_biv_lok}, $\cT_\smth$ is a Fr\'echet-Arens-Michael algebra.

Another natural example is the {\em holomorphic Toeplitz algebra} $\cT_\hol$ recently introduced
by Panarin \cite{Panarin}. The definition of $\cT_\hol$ is similar to that of $\cT_\smth$.
Namely, we consider the algebra
\[
\cK_\hol=\Bigl\{ a=(a_{ij})_{i,j\in\Z_+} : a_{ij}\in\CC,\;
\| a\|_n=\sum_{i,j} |a_{ij}| n^{i+j}<\infty\;\forall n\in\Z_+\Bigr\}
\]
of {\em holomorphic compact operators}, and define $\cT_\hol$ as follows:
\[
\cT_\hol=\{ T_f + K : f\in \cO(\CC^\times),\; K\in\cK_\hol \}.
\]
Here $\cO(\CC^\times)$ denotes the algebra of holomorphic functions
on $\CC^\times=\CC\setminus\{ 0\}$.
Similarly to $\cT_\smth$, $\cT_\hol$ becomes a Fr\'echet-Arens-Michael algebra if we identify
the underlying vector space of $\cT_\hol$ with $\cO(\CC^\times)\oplus \cK_\hol$
via the isomorphism $(f,K)\mapsto T_f+K$.

The following result is probably due to Meyer
(see \cite{Meyer_locan}, where a much more general result is proved).
Since this is the main motivation for the present paper, we give a proof here
for the reader's convenience. This elementary proof is due to Aristov
(private communication).

\begin{prop}
\label{prop:Talg_qfree}
$\cT_\alg$ is quasi-free.
\end{prop}
\begin{proof}
By Remark \ref{rem:qfree_strongest}, it suffices to show that $\cT_\alg$ is quasi-free
in the purely algebraic sense. Let
\begin{equation}
\label{ext_Talg}
0\to I \to R\xra{p}\cT_\alg\to 0
\end{equation}
be a square-zero extension of $\cT_\alg$.
Choose $a,b\in R$ such that $p(a)=u$ and $p(b)=v$.
Then $ab=1+c$ for some $c\in I$. Since $I^2=0$,
we have $ab(1-c)=1-c^2=1$.
Letting $b'=b(1-c)$, we see that $ab'=1$.
By the universal property of $\cT_\alg$, there is a unique homomorphism
$j\colon\cT_\alg\to R$ such that
$j(u)=a$ and $j(v)=b'$.
Since $p(a)=u$ and $p(b')=p(b-bc)=p(b)=v$, we have $pj=\id_{\cT_\alg}$.
Thus \eqref{ext_Talg} splits, which completes the proof.
\end{proof}

\section{A family of locally convex Toeplitz algebras}
\label{sect:TPQ}
In this section we construct a family $\{\cT_{P,Q}\}$ of locally convex subalgebras
of $\cT$, where $P$ and $Q$ are K\"othe sets satisfying some natural conditions.
We will also show that $\cT_\alg$, $\cT_\smth$, and $\cT_\hol$ are special cases
of our construction.

Let $I$ be any set, and let $P$ be a set of nonnegative real-valued functions on $I$.
For $p\in P$ and $i\in I$, we write $p_i$ for $p(i)$.
Recall that $P$ is a {\em K\"othe set} on $I$ if the following axioms are
satisfied (see, e.g., \cite{Pietsch}):
\begin{align*}
\tag*{(P1)}
&\forall\,i\in I\quad\exists\, p\in P\quad p_i>0\, ;\\
\tag*{(P2)}
&\forall\, p,q\in I\quad\exists\, r\in P\quad\max\{ p_i,q_i\}\le r_i\quad (i\in I).
\end{align*}
Given a K\"othe set $P$, the {\em K\"othe space}
$\lambda(P)=\lambda(I,P)$ is defined as follows [loc. cit.]:
\begin{equation*}
\lambda(P)=
\Bigl\{ x=(x_i)\in \CC^I :
\| x\|_p=\sum_i |x_i|p_i <\infty\quad\forall\, p\in P\Bigr\}\, .
\end{equation*}
This is a complete locally convex space
with the topology determined by
the family $\{\|\cdot\|_p : p\in P\}$ of seminorms. Clearly, $\lambda(P)$
is a Fr\'echet space if and only if $P$ contains an at most countable
cofinal subset.

For each $i\in I$ denote by $e_i$ the function on $I$ which is $1$ at $i$ and
$0$ elsewhere. Obviously, $x=\sum_i x_i e_i$ for each $x\in\lambda(P)$.

Given K\"othe sets $P\subset [0,+\infty)^I$ and $Q\subset [0,+\infty)^J$, let
$P\times Q$ denote the K\"othe set on $I\times J$ consisting of all functions
of the form $(i,j)\mapsto p_i q_j$ ($p\in P$, $q\in Q$).
By \cite{Pietsch2}, there exists a topological isomorphism
\begin{equation}
\label{Kothe_Ptens}
\lambda(P)\Ptens\lambda(Q)\lriso\lambda(P\times Q),\qquad
e_i\otimes e_j\mapsto e_{(i,j)}.
\end{equation}
Moreover, if we identify $\lambda(P)\Ptens\lambda(Q)$ with $\lambda(P\times Q)$
via \eqref{Kothe_Ptens}, then for each $p\in P$ and $q\in Q$
the seminorm $\|\cdot\|_{(p,q)}$ on $\lambda(P\times Q)$
is identified with the projective tensor seminorm
$\|\cdot\|_p\otimes_\pi\|\cdot\|_q$ on $\lambda(P)\Ptens\lambda(Q)$.

From now on, we concentrate on K\"othe sets on $\Z_+=\{ 0,1,2,\ldots\}$.
As a first step towards constructing the locally convex Toeplitz algebras $\cT_{P,Q}$,
we define auxiliary power series algebras, which will play the role of
``building blocks'' for $\cT_{P,Q}$. More exactly, we would like to make
$\lambda(P)$ into an algebra under convolution. This requires
one more condition on the K\"othe set $P$.

\begin{definition}
Let $P$ be a K\"othe set on $\Z_+$. We say that $P$ is a {\em weighted set}
if for each $p\in P$ we have $p_0=1$,
and if for each $p\in P$ there exist $p'\in P$ and $C>0$ such that
\[
p_{i+j}\le Cp'_i p'_j \qquad (i,j\in\Z_+).
\]
If the above property holds with $p'=p$ and $C=1$, then we say that
$P$ is an {\em $m$-weighted set}.
\end{definition}

\begin{prop}
For each weighted set $P$ on $\Z_+$
there exists a unique multiplication on $\lambda(P)$ satisfying $e_{i+j}=e_i e_j$ ($i,j\in\Z_+$)
and making $\lambda(P)$ into a $\Ptens$-algebra.
If, moreover, $P$ is an $m$-weighted set, then each seminorm
$\|\cdot\|_p$ is submultiplicative, and so $\lambda(P)$ is an Arens-Michael algebra.
\end{prop}

The proof is straightforward and is therefore omitted. Observe also that $\lambda(P)$ is unital
and that $e_0$ is the identity of $\lambda(P)$.

For our purposes, it will be convenient to let $z=e_1$ and to
denote the algebra $\lambda(P)$ with the above
multiplication by $\lambda(z,P)$ (so that $z$ plays the role of a ``formal variable''). Thus we have
\begin{equation*}
\lambda(z,P)=
\Bigl\{ a=\sum_{i=0}^\infty a_i z^i : a_i\in\CC,\;
\| a\|_p=\sum_i |a_i|p_i <\infty\quad\forall\, p\in P\Bigr\}\, .
\end{equation*}
Obviously, we have a chain of algebra embeddings $\CC[z]\subset\lambda(z,P)\subset\CC[[z]]$,
and both embeddings are dense.

\begin{example}
Let $P=\{ p^{(1)},\, p^{(2)},\ldots\}$, where $p^{(k)}=(1,\ldots ,1,0,0,\ldots)$
(with $1$ repeated $k$ times). Clearly, $P$ is an $m$-weighted set, and we have
$\lambda(z,P)=\CC[[z]]$, both algebraically and topologically.
\end{example}

\begin{example}
\label{example:C^inf_power}
Let $P=\{ p^{(1)},\, p^{(2)},\ldots\}$, where $p^{(k)}_n=(1+n)^k$ ($n\in\Z_+$).
It is easily seen that $P$ is an $m$-weighted set, and that $\lambda(z,P)$ is topologically
isomorphic (via the Fourier transform) to the subalgebra of $C^\infty(\TT)$
consisting of those smooth functions whose negative Fourier coefficients vanish.
\end{example}

\begin{example}
\label{example:O_power}
Let $R\in (0,+\infty]$, and let $P=\{ p^{(r)} : 0<r<R\}$, where $p^{(r)}_n=r^n$ for all $n\in\Z_+$.
Obviously, $P$ is an $m$-weighted set, and $\lambda(z,P)$ is nothing but the algebra
$\cO(\DD_R)$ of holomorphic functions on the open disk $\DD_R=\{ z\in\CC : |z|<R\}$.
\end{example}

\begin{example}
\label{example:poly_power}
Let $P$ denote the collection of all positive sequences $p\in (0,+\infty)^{\Z_+}$ satisfying $p_0=1$.
Given $p\in P$, we can easily construct $p'\in P$
such that
\begin{equation}
\label{pij}
p_{i+j}\le p'_i p'_j
\end{equation}
for all $i,j\in\Z_+$. Indeed, let $p'_0=1$, and assume that
we have already constructed $p'_0,\ldots ,p'_k$ satisfying \eqref{pij} whenever $i,j\le k$.
Clearly, there exists $p'_{k+1}>0$ such that $p_{i+k+1}\le p'_i p'_{k+1}$ for all $i\le k+1$.
Hence \eqref{pij} holds for all $i,j\le k+1$, and the induction argument completes the proof.
Thus $P$ is a weighted set. We clearly have $\lambda(z,P)=\CC[z]$, and the topology
on $\lambda(z,P)$ determined by $P$ is the strongest locally convex topology.
It is well known that $\CC[z]$ is not locally $m$-convex, and so $P$ is not
an $m$-weighted set.
\end{example}

Now we are ready to define our family $\{\cT_{P,Q}\}$ of locally convex Toeplitz algebras.
Let us identify the underlying
vector space of $\cT_\alg$ with $\CC[v]\Tens\CC[u]$ by sending each monomial
$v^i u^j$ to $v^i\otimes u^j$ ($i,j\in\Z_+$).
Given weighted sets $P$ and $Q$, we let $\cT_{P,Q}=\lambda(v,Q)\Ptens\lambda(u,P)$.
Clearly, $\cT_\alg$ is a dense vector subspace of $\cT_{P,Q}$.
By using \eqref{Kothe_Ptens}, we see that
\[
\cT_{P,Q}=\Bigl\{ a=\sum_{i,j\in\Z_+} c_{ij} v^i u^j :
c_{ij}\in\CC,\; \| a\|_{q,p}=\sum_{i,j\in\Z_+} |c_{ij}|q_i p_j<\infty
\;\forall q\in Q,\; \forall p\in P\Bigr\}.
\]
In order to make $\cT_{P,Q}$ into a $\Ptens$-algebra containing $\cT_\alg$ as a subalgebra,
we need to impose one more condition on the K\"othe sets $P$ and $Q$.

\begin{definition}
Let $P$ be a weighted set. We say that $P$ is {\em monotone} if for each $p\in P$
and each $i\in\Z_+$ we have $p_i\le p_{i+1}$.
\end{definition}

\begin{prop}
For each pair $P,Q$ of monotone weighted sets on $\Z_+$, there exists a unique jointly continuous
multiplication on $\cT_{P,Q}$ that extends the multiplication on $\cT_\alg$.
If, moreover, $P$ and $Q$ are $m$-weighted sets,
then $\cT_{P,Q}$ is an Arens-Michael algebra.
\end{prop}
\begin{proof}
Given $p\in P$ and $q\in Q$, choose $p'\in P$, $q'\in Q$, and $C>0$
such that for each $i,j\in\Z_+$ we have
\[
p_{i+j}\le C p'_i p'_j,\qquad q_{i+j}\le C q'_i q'_j.
\]
Without loss of generality, we may assume that $p\le p'$ and $q\le q'$.
For each pair $v^i u^j$ and $v^k u^\ell$ of monomials in $\cT_\alg$, we clearly have
\begin{equation}
\label{toepl_bas_rel}
(v^i u^j)(v^k u^\ell)=
\begin{cases}
v^i u^{j-k+\ell} & \text{if $j\ge k$;}\\
v^{i+k-j} u^\ell & \text{if $j\le k$.}
\end{cases}
\end{equation}
If $j\ge k$, then
\[
\begin{split}
\| (v^i u^j)(v^k u^\ell)\|_{q,p}
&=\| v^i u^{j-k+\ell} \|_{q,p}
=q_i p_{j-k+\ell}\\
&\le q_i p_{j+\ell}
\le C q'_i q'_k p'_j p'_\ell
= C \| v^i u^j\|_{q',p'} \| v^k u^\ell\|_{q',p'}.
\end{split}
\]
On the other hand, if $j\le k$, then
\[
\begin{split}
\| (v^i u^j)(v^k u^\ell)\|_{q,p}
&=\| v^{i+k-j} u^\ell \|_{q,p}
=q_{i+k-j} p_\ell\\
&\le q_{i+k} p_\ell
\le C q'_i q'_k p'_j p'_\ell
= C \| v^i u^j\|_{q',p'} \| v^k u^\ell\|_{q',p'}.
\end{split}
\]
Thus for each pair of monomials $a,b\in\cT_\alg$ we have
\begin{equation}
\label{mult_cont_toepl}
\| ab\|_{q,p} \le C \| a\|_{q',p'} \| b\|_{q',p'}.
\end{equation}
Observe now that, if we take any $a\in\cT_\alg$ and decompose it as a finite sum $a=\sum_i a_i$
of linearly independent monomials $a_i$, then we have
\begin{equation}
\label{norm_sum_monom}
\| a\|_{q,p}=\sum_i \| a_i\|_{q,p}.
\end{equation}
This implies that \eqref{mult_cont_toepl} actually holds for each $a,b\in\cT_\alg$.
Therefore the multiplication on $\cT_\alg$ is jointly continuous for the topology
inherited from $\cT_{P,Q}$, and so it uniquely extends by continuity to $\cT_{P,Q}$.
Finally, if both $P$ and $Q$ are $m$-weighted sets, then we can repeat the above argument
with $p'=p$, $q'=q$, and $C=1$, which implies that each seminorm $\|\cdot\|_{q,p}$
is submultiplicative. This completes the proof.
\end{proof}

Let us now show that $\cT_\smth$ and $\cT_\hol$ are special cases of the above construction.

\begin{prop}
\label{prop:T_smth_PQ}
Let $P=\{ p^{(1)},\, p^{(2)},\ldots\}$, where $p^{(k)}_n=(1+n)^k$ for all $k\in\N$,
$n\in\Z_+$ (see Example~{\upshape\ref{example:C^inf_power}}). Then the identity map of $\cT_\alg$
uniquely extends to a topological algebra isomorphism $\cT_{P,P}\cong\cT_\smth$.
\end{prop}
\begin{proof}
For each $i,j\in\Z_+$, let $e_{i,j}\in M_\infty$ denote the respective matrix unit
(i.e., the matrix whose $(i,j)$th entry is $1$ and the other entries are $0$).
The restriction of \eqref{CK=T} to $\CC[z,z^{-1}]\oplus M_\infty$
is a vector space isomorphism between $\CC[z,z^{-1}]\oplus M_\infty$ and $\cT_\alg$,
which acts on the basis $\{ z^k,\; e_{i,j} : k\in\Z,\; i,j\in\Z_+\}$ as follows (see \cite{Cuntz_cycl}):
\begin{equation}
\label{zevu}
z^k \mapsto v^k \quad (k\ge 0), \qquad
z^{-k} \mapsto u^k \quad (k\ge 0),\qquad
e_{i,j}\mapsto v^i(1-vu)u^j.
\end{equation}
Identifying the underlying vector spaces of $\CC[z,z^{-1}]\oplus M_\infty$ and $\cT_\alg$
via the above isomorphism, we conclude from \eqref{zevu} that
\begin{equation}
\label{v^iu^j}
v^i u^j=
\begin{cases}
z^{i-j}-(e_{i-j,0}+e_{i-j+1,1}+\cdots +e_{i-1,j-1}) & \text{if $i\ge j$},\\
z^{i-j}-(e_{0,j-i}+e_{1,j-i+1}+\cdots +e_{i-1,j-1}) & \text{if $i<j$}.
\end{cases}
\end{equation}
By definition, the topology on $\cT_\alg$ inherited from $\cT_{P,P}$ is given by the family
$\{\|\cdot\|_k : k\in\Z_+\}$ of norms, where
\[
\Bigl\| \sum_{i,j\in\Z_+} c_{ij} v^i u^j\Bigr\|_k
=\sum_{i,j\in\Z_+} |c_{ij}| (1+i)^k (1+j)^k \qquad (c_{ij}\in\CC).
\]
Since $\cT_\smth=C^\infty(\TT)\oplus\cK_\smth$ as locally convex spaces,
we see that the topology on $\cT_\alg$ inherited from $\cT_\smth$
is given by the family $\{\|\cdot\|'_k : k\in\Z_+\}$ of norms, where
\[
\Bigl\| \sum_{i,j\in\Z_+} a_{ij} e_{i,j}+\sum_{p\in\Z} b_p z^p \Bigr\|'_k
=\sum_{i,j\in\Z_+} |a_{ij}| (1+i+j)^k
+\sum_{p\in\Z} |b_p|(1+|p|)^k.
\]
Thus, to complete the proof, it suffices to show that the families
\begin{equation}
\label{norms_T_alg}
\{\|\cdot\|_k : k\in\Z_+\}\quad\text{and}\quad \{\|\cdot\|'_k : k\in\Z_+\}
\end{equation}
of norms are equivalent
on $\cT_\alg$.

If $i,j\in\Z_+$ and $i\ge j$, then we see from \eqref{v^iu^j} that
\[
\begin{split}
\| v^i u^j\|'_k
&=(1+i-j)^k+(1+i-j)^k+(3+i-j)^k+\cdots +(-1+i+j)^k\\
&\le (j+1)(1+i+j)^k \le (1+i+j)^{k+1} \le (1+i)^{k+1} (1+j)^{k+1}
=\| v^i u^j\|_{k+1}.
\end{split}
\]
A similar argument shows that $\| v^i u^j\|'_k\le\| v^i u^j\|_{k+1}$ whenever
$i,j\in\Z_+$ and $i<j$. Now for each $a=\sum_{i,j} c_{ij} v^i u^j\in\cT_\alg$ we have
\begin{equation}
\label{k_k+1}
\| a\|'_k \le \sum_{i,j} |c_{ij}| \| v^i u^j\|'_k
\le \sum_{i,j} |c_{ij}| \| v^i u^j\|_{k+1}
=\| a\|_{k+1}.
\end{equation}
On the other hand, for each $i,j\in\Z_+$ we have
\begin{equation}
\label{e_ij_k}
\begin{split}
\| e_{i,j}\|_k
&=\| v^i u^j-v^{i+1} u^{j+1}\|_k
=(1+i)^k (1+j)^k + (2+i)^k (2+j)^k\\
&\le (1+i+j)^{2k} +4^k (1+i+j)^{2k}
=(4^k+1) \| e_{i,j}\|'_{2k}.
\end{split}
\end{equation}
Also, for each $p\in\Z_+$ we have
\[
\| z^p\|_k = \| v^p\|_k =(1+p)^k =  \| z^p \|'_k,
\]
and, similarly, $\| z^{-p}\|_k = \| z^{-p}\|'_k$. Together with \eqref{e_ij_k}, this implies
that for each $a=\sum_{i,j} a_{ij} e_{i,j}+\sum_p b_p z^p\in\cT_\alg$ we have
\[
\| a\|_k \le (4^k+1) \sum_{i,j} |a_{ij}| \| e_{i,j}\|'_{2k} + \sum_p |b_p| \| z^p\|'_{2k}
\le (4^k+1) \| a\|'_{2k}.
\]
Comparing this with \eqref{k_k+1}, we conclude that the families \eqref{norms_T_alg}
are indeed equivalent. This completes the proof.
\end{proof}

\begin{prop}
\label{prop:T_hol_PQ}
Let $P=\{ p^{(1)},\, p^{(2)},\ldots\}$, where $p^{(k)}_n=k^n$ for all $k\in\N$,
$n\in\Z_+$ (see Example~{\upshape\ref{example:O_power}}). Then the identity map of $\cT_\alg$
uniquely extends to a topological algebra isomorphism $\cT_{P,P}\cong\cT_\hol$.
\end{prop}

We omit the proof, because it is similar to that of Proposition~\ref{prop:T_smth_PQ}.

\begin{remark}
Let $P$ denote the collection of all positive nondecreasing sequences $p\in (0,+\infty)^{\Z_+}$
satisfying $p_0=1$. It easily follows from Example~\ref{example:poly_power} that
$\cT_{P,P}=\cT_\alg$, both algebraically and topologically.
\end{remark}

\begin{remark}
Proposition~\ref{prop:T_hol_PQ} easily implies that $\cT_\hol$ is the Arens-Michael envelope
of $\cT_\alg$ (see, e.g., \cite[Chap. V]{X1} or \cite{Pir_qfree} for general information
on Arens-Michael envelopes).
In other words, $\cT_\hol$ is the universal Arens-Michael algebra with two distinguished
elements $u$ and $v$ satisfying $uv=1$ (cf. \cite[Theorem~2.14]{Panarin}).
For a similar characterization of $\cT_\smth$, see \cite[Satz~6.1]{Cuntz_biv_lok}.
\end{remark}

\section{Calculation of $\Omega^1\cT_\alg$}

In this section, we calculate explicitly the bimodule $\Omega^1\cT_\alg$ of noncommutative
differential $1$-forms over the algebraic Toeplitz algebra. This result will be applied in
Section~\ref{section:main} to finding sufficient conditions for the quasi-freeness
of $\cT_{P,Q}$.

\begin{lemma}
\label{lemma:hom_idemp}
Let $A$ be an algebra and $p\in A$ an idempotent. Then for each left $A$-module $M$
there exists a vector space isomorphism
\begin{equation}
\label{hom_idemp}
{_A}\h(Ap,M)\lriso pM, \qquad \varphi\mapsto\varphi(p).
\end{equation}
\end{lemma}
\begin{proof}
For each $\varphi\in {_A}\h(Ap,M)$ we have $(1-p)\varphi(p)=\varphi(0)=0$, hence $\varphi(p)\in pM$.
Now observe that for each $m\in pM$ we have an $A$-module morphism
$\varphi_m\colon Ap\to M$ given by $\varphi_m(ap)=am$. Indeed, if $ap=0$, then
$am=(a-ap)m=a(1-p)m=0$, so $\varphi_m$ is well defined. Now it is easy to see that
the map $m\mapsto\varphi_m$ is the inverse of \eqref{hom_idemp}.
\end{proof}

From now on, we let $e'=vu\in\cT_\alg$ and $e=1-vu\in\cT_\alg$. Clearly, both $e'$ and $e$ are
idempotents.

\begin{theorem}
\label{thm:Omega_Talg}
Let $A=\cT_\alg$. There exists an $A$-bimodule isomorphism
\begin{equation}
\label{Omega_Talg}
\Omega^1 A \cong (A\Tens A)\oplus (Ae\Tens A).
\end{equation}
Under the above identification, the universal derivation
$d_A\colon A\to (A\Tens A)\oplus (Ae\Tens A)$ acts as follows:
\begin{equation}
\label{d_T_alg}
d_A(u)=(1\otimes 1,0),\qquad d_A(v)=(-v\otimes v,e\otimes 1).
\end{equation}
\end{theorem}
\begin{proof}
For each $A$-bimodule $M$, every derivation $d\colon A\to M$ is uniquely determined
by the elements $m=du$ and $n=dv$. Conversely, given $(m,n)\in M\times M$,
we conclude from \eqref{der=hom} that a derivation $d\colon A\to M$ taking
$u$ to $m$ and $v$ to $n$ exists if and only if we have $(u,m)(v,n)=(1,0)$
in $A\times M$, i.e., if and only if
\begin{equation}
\label{unmv}
un+mv=0.
\end{equation}
Thus we have a vector space isomorphism
\begin{equation}
\label{Der_MM}
\Der(A,M)\cong\Bigl\{\, (m,n)\in M\times M : un+mv=0\,\Bigr\},\quad d\mapsto (du,dv).
\end{equation}
Now observe that for each $m\in M$ the element $n=-vmv$ satisfies \eqref{unmv}.
Also, for each $\ell\in eM$ we have $u\ell=uvu\ell=ue'\ell=0$, so $n=-vmv+\ell$ also
satisfies \eqref{unmv}.
Thus we have a linear map
\[
\varphi\colon M\oplus eM\to \Bigl\{\, (m,n)\in M\times M : un+mv=0\,\Bigr\},
\qquad (m,\ell)\mapsto (m,-vmv+\ell).
\]
On the other hand, if $(m,n)\in M\times M$ satisfies \eqref{unmv}, then
\[
e'(n+vmv)=vun+vuvmv=v(un+mv)=0,
\]
so $n+vmv\in eM$, and we have a linear map
\[
\psi\colon \Bigl\{\, (m,n)\in M\times M : un+mv=0\,\Bigr\} \to M\oplus eM,
\qquad (m,n)\mapsto (m,n+vmv).
\]
Clearly, $\varphi\psi=\id$ and $\psi\varphi=\id$, whence $\varphi$ and $\psi$ are
isomorphisms. Thus we have
\begin{equation}
\label{MeM}
\Bigl\{\, (m,n)\in M\times M : un+mv=0\,\Bigr\} \cong M\oplus eM.
\end{equation}
Let now $p=e\otimes 1\in A^\env$. Since $p$ is an idempotent, Lemma~\ref{lemma:hom_idemp}
yields vector space isomorphisms
\begin{equation}
\label{MeM2}
\begin{split}
M\oplus eM
&=M\oplus pM
\cong {_{A^\env}}\h(A^\env,M)\oplus {_{A^\env}}\h(A^\env p,M)\\
&\cong {_{A^\env}}\h(A^\env\oplus A^\env p,M)
\cong {_A}\h_A\bigl((A\Tens A)\oplus (Ae\Tens A),M\bigr).
\end{split}
\end{equation}
Since all the above isomorphisms are natural in $M$, we conclude from
\eqref{Der_MM}, \eqref{MeM}, and \eqref{MeM2} that
$\Omega^1 A \cong (A\Tens A)\oplus (Ae\Tens A)$, as required.

To complete the proof, let $M=(A\Tens A)\oplus (Ae\Tens A)$, and observe that
the universal derivation $d_A\colon A\to M$ corresponds to the identity map $\id_M$ under
the composition of isomorphisms \eqref{Der_MM}, \eqref{MeM}, and \eqref{MeM2}.
Reading \eqref{MeM2} from right to left, we see the pair $(m,\ell)\in M\oplus eM$
corresponding to $\id_M$ is given by $m=(1\otimes 1,0)$ and $\ell=(0,e\otimes 1)$.
Hence
\[
(d_A u,d_A v)=\varphi(m,\ell)=(m,-vmv+\ell)=\bigl((1\otimes 1,0),(-v\otimes v,e\otimes 1)\bigr),
\]
as required.
\end{proof}

\section{Homological epimorphisms from quasi-free algebras}

Our next goal is to give a convenient criterion for a $\Ptens$-algebra homomorphism
$\varphi\colon A\to B$ to be a homological epimorphism, assuming that $A$ is quasi-free
and that $\varphi$ is an epimorphism. In the next section, we apply this result to
the embedding of $\cT_\alg$ into $\cT_{P,Q}$.

Let $A$ and $B$ be $\Ptens$-algebras, and let $\varphi\colon A\to B$ be a continuous homomorphism.
The universal property of $\Omega^1 A$ yields an $A$-$\Ptens$-bimodule morphism
$\Omega^1 A\to\Omega^1 B$ uniquely determined by $d_A a\mapsto d_B(\varphi(a))$.
Tensoring $\Omega^1 A$ by $B$ on both sides, we get a $B$-$\Ptens$-bimodule morphism
\begin{equation}
\label{alpha}
\alpha(\varphi)\colon B\ptens{A}\Omega^1 A\ptens{A} B\to\Omega^1 B,
\qquad 1\otimes d_A a\otimes 1 \mapsto d_B(\varphi(a)).
\end{equation}

\begin{theorem}
\label{thm:qfree_homepi}
Let $\varphi\colon A\to B$ be a $\Ptens$-algebra epimorphism. Suppose that $A$ is quasi-free.
Then the following conditions are equivalent:
\begin{mycompactenum}
\item
$\alpha(\varphi)$ is an isomorphism;
\item
there exists a derivation $D\colon B\to B\ptens{A}\Omega^1 A\ptens{A} B$ making the
diagram
\begin{equation}
\label{D}
\xymatrix{
A \ar[r]^{d_A} \ar[d]_\varphi & \Omega^1 A \ar[d]^\beta\\
B \ar[r]^(.3)D & B\ptens{A}\Omega^1 A\ptens{A} B
}
\end{equation}
commute, where $\beta$ is given by $\omega\mapsto 1\otimes\omega\otimes 1$;
\item
$\varphi$ is a homological epimorphism.
\end{mycompactenum}
If the above conditions are satisfied, then $B$ is quasi-free.
\end{theorem}

To prove Theorem \ref{thm:qfree_homepi}, we need the following simple lemma.

\begin{lemma}
\label{lemma:der_epi}
Let $\varphi\colon A\to B$ be a $\Ptens$-algebra epimorphism, and let
$M$ be a $B$-$\Ptens$-bimodule. Suppose that $d_1,d_2\colon B\to M$ are derivations
such that $d_1\varphi=d_2\varphi$. Then $d_1=d_2$.
\end{lemma}
\begin{proof}
Define $\Ptens$-algebra homomorphisms $\psi_1,\psi_2\colon B\to B\times M$
by $\psi_i(b)=(b,d_i(b))$ ($b\in B$, $i=1,2$); see~\eqref{der=hom}. We clearly have
$\psi_1\varphi=\psi_2\varphi$, whence $\psi_1=\psi_2$ and $d_1=d_2$.
\end{proof}

\begin{proof}[Proof of Theorem {\upshape\ref{thm:qfree_homepi}}]
$\mathrm{(i)}\Longrightarrow\mathrm{(ii)}$.
Define a derivation $D\colon B\to B\ptens{A}\Omega^1 A\ptens{A} B$ by $D=\alpha^{-1} d_B$,
where $\alpha=\alpha(\varphi)$. To show that \eqref{D} commutes, it suffices to prove that
$\alpha D\varphi=\alpha \beta d_A$, which is equivalent to $d_B\varphi=\alpha \beta d_A$.
However, the latter relation is precisely the definition of $\alpha$ (see \eqref{alpha}).

$\mathrm{(ii)}\Longrightarrow\mathrm{(i)}$.
The universal property of $\Omega^1 B$ yields a $B$-$\Ptens$-bimodule morphism
\[
\tau\colon\Omega^1 B\to B\ptens{A}\Omega^1 A\ptens{A} B,\qquad
\tau d_B=D.
\]
We claim that $\tau$ is the inverse of $\alpha$. Indeed, for each $a\in A$ we have
\[
(\tau\alpha)(1\otimes d_A a\otimes 1)
=\tau(d_B(\varphi(a)))
=D(\varphi(a))
=1\otimes d_A a\otimes 1.
\]
Since $B\ptens{A}\Omega^1 A\ptens{A} B$ is generated (as a $B$-$\Ptens$-bimodule)
by elements of the form $1\otimes d_A a\otimes 1$ ($a\in A$), we conclude that $\tau\alpha=\id$.
To prove that $\alpha\tau=\id$, it suffices to show that $\alpha\tau d_B=d_B$
(by the universal property of $\Omega^1 B$). Since both $d_B$ and $\alpha\tau d_B$
are derivations, Lemma~\ref{lemma:der_epi} implies that $\alpha\tau d_B=d_B$
whenever $\alpha\tau d_B\varphi=d_B\varphi$. For each $a\in A$, we have
\[
(\alpha\tau d_B\varphi)(a)=(\alpha D \varphi)(a)=(\alpha\beta d_A)(a)=(d_B\varphi)(a).
\]
In view of the above remarks, this proves that $\alpha\tau=\id$. Hence $\alpha$ is an isomorphism.

$\mathrm{(i)}\iff\mathrm{(iii)}$.
Since $A$ is quasi-free, we see that \eqref{Omega^1_can} is a projective resolution
of $A$ in $A\bimod A$.
Tensoring \eqref{Omega^1_can} by $B$ on both sides, we conclude
that $\varphi$ is a homological epimorphism if and only if the resulting sequence
\begin{equation}
\label{B-Omega-B}
\xymatrix{
0 \ar[r] & B\ptens{A}\Omega^1 A\ptens{A} B \ar[r]^(.65){k_B}
& B\Ptens B \ar[r]^(.6){\mu_B}
& B \ar[r] & 0
}
\end{equation}
is admissible (where $k_B$ is induced by $j_A$ in the obvious way).
Clearly, \eqref{B-Omega-B} fits into the commutative diagram
\begin{equation}
\label{B-Omega-B2}
\xymatrix{
0 \ar[r] & B\ptens{A}\Omega^1 A\ptens{A} B \ar[r]^(.65){k_B} \ar[d]_\alpha
& B\Ptens B \ar[r]^(.6){\mu_B} \ar@{=}[d]
& B \ar[r] \ar@{=}[d] & 0\\
0 \ar[r] & \Omega^1 B \ar[r]^{j_B} & B\Ptens B \ar[r]^(.6){\mu_B}
& B \ar[r] & 0
}
\end{equation}
Since the bottom row of \eqref{B-Omega-B2} is admissible, we conclude that
\eqref{B-Omega-B} is admissible if and only if $\alpha$ is an isomorphism.

To complete the proof, recall that (iii) together with the assumption that $A$ is quasi-free
implies that $B$ is quasi-free as well (see Remark~\ref{rem:homepi_db}).
\end{proof}

\begin{remark}
Observe that we have $\mathrm{(i)}\Longrightarrow\mathrm{(ii)}$ for
any $\Ptens$-algebra homomorphism, and
$\mathrm{(ii)}\Longrightarrow\mathrm{(i)}$ for any $\Ptens$-algebra epimorphism
(see the proof of Theorem~\ref{thm:qfree_homepi}). In both cases, the assumption that
$A$ is quasi-free is inessential.
\end{remark}

\section{The main result}
\label{section:main}

Given two sequences $p,q\in [0,+\infty)^{\Z_+}$, we define their {\em convolution}
$p*q\in [0,+\infty)^{\Z_+}$ by
\[
(p*q)_k=\sum_{i+j=k} p_i q_j \qquad (k\in\Z_+).
\]
If $P$ and $Q$ are K\"othe sets on $\Z_+$, we let
\[
P*Q=\{ p*q : p\in P,\; q\in Q\}.
\]
Clearly, $P*Q$ is a K\"othe set as well. Finally, we say that $P$ is {\em dominated} by $Q$
and write $P\prec Q$ if for each $p\in P$ there exist $q\in Q$ and $C>0$ such that
$p_i\le Cq_i$ ($i\in\Z_+$).

\begin{theorem}
\label{thm:main}
Let $P$ and $Q$ be monotone weighted sets on $\Z_+$ such that
$P*P\prec P$ and $Q*Q\prec Q$. Then the embedding of $\cT_\alg$
into $\cT_{P,Q}$ is a homological epimorphism. As a consequence, $\cT_{P,Q}$ is quasi-free.
\end{theorem}
\begin{proof}
Let $A=\cT_\alg$ and $B=\cT_{P,Q}$. To prove the result, we will show that the embedding
of $A$ into $B$ satisfies condition (ii) of Theorem~\ref{thm:qfree_homepi}.
Using Theorem~\ref{thm:Omega_Talg}, we see that diagram~\eqref{D} looks as follows:
\begin{equation*}
\xymatrix{
A \ar[r]^(.25){d_A} \ar@{^{(}->}[d] &   (A\Tens A)\oplus (Ae\Tens A) \ar@{^{(}->}[d]\\
B \ar[r]^(.25)D & (B\Ptens B)\oplus (Be\Ptens B)
}
\end{equation*}
Hence, to show that $D$ exists, it suffices to prove the continuity of $d=d_A$ for
the topologies on $A$ and $(A\Tens A)\oplus (Ae\Tens A)$
inherited from $B$ and $(B\Ptens B)\oplus (Be\Ptens B)$, respectively.

Using \eqref{Kothe_Ptens}, we can identify the underlying topological vector space
of $B\Ptens B$ with $\lambda(Q\times P\times Q\times P)$. Since $Be$ is a topological
direct summand of $B$, we can interpret $Be\Ptens B$ as a subspace of $B\Ptens B$.
Given $q\in Q$ and $p\in P$, define a seminorm $\|\cdot\|'_{q,p}$ on
$(B\Ptens B)\oplus (Be\Ptens B)$ by
\[
\| (c,d)\|'_{q,p}=\| c\|_{q,p,q,p}+\| d\|_{q,p,q,p}
\qquad (c\in B\Ptens B,\; d\in Be\Ptens B).
\]
Clearly, $\{\|\cdot\|'_{q,p} : q\in Q,\; p\in P\}$ is a defining family of seminorms
on $(B\Ptens B)\oplus (Be\Ptens B)$.

Let $i,j\in\Z_+$. Using \eqref{d_T_alg}, we see that
\begin{equation}
\label{dvu}
\begin{split}
d(v^i)u^j
&=\sum_{k=0}^{i-1} v^k\cdot dv\cdot v^{i-k-1}u^j
=\sum_{k=0}^{i-1} (-v^{k+1}\otimes v^{i-k}u^j\, ,\, v^k e\otimes v^{i-k-1}u^j)\\
&=\sum_{k=0}^{i-1} (-v^{k+1}\otimes v^{i-k}u^j\, ,\,
v^k\otimes v^{i-k-1}u^j-v^{k+1}u\otimes v^{i-k-1}u^j).
\end{split}
\end{equation}
Similarly,
\begin{equation}
\label{vdu}
v^i d(u^j)
=\sum_{\ell=0}^{j-1} v^i u^\ell\cdot du\cdot u^{j-\ell-1}
=\sum_{\ell=0}^{j-1} (v^i u^\ell\otimes u^{j-\ell-1},0).
\end{equation}
Take any $p\in P$, $q\in Q$ and find $C>0$, $p'\in P$, and $q'\in Q$ such that
$p*p\le Cp'$, $q*q\le Cq'$. Choose also $q''\in Q$ and $C_1>0$ such that
$q'_{k+\ell}\le C_1 q''_k q''_\ell$ ($k,\ell\in\Z_+$).
Without loss of generality, we may also assume that
$p\le p'$ and $q\le q'\le q''$. Using \eqref{dvu}, we obtain
\begin{equation}
\label{dvu2}
\begin{split}
\| d(v^i)u^j\|'_{q,p}
&=\sum_{k=0}^{i-1} (q_{k+1} q_{i-k} p_j + q_k q_{i-k-1} p_j + q_{k+1} p_1 q_{i-k-1} p_j)\\
&\le p_j\bigl((q*q)_{i+1}+(q*q)_{i-1}+p_1(q*q)_i\bigr)\\
&\le Cp_j(q'_{i+1}+q'_{i-1}+p_1 q'_i)\\
&\le Cp_j(C_1 q''_1 q''_i + q''_i + p_1 q''_i)=
C_2 \| v^i u^j\|_{q'',p}\, ,
\end{split}
\end{equation}
where $C_2=C(C_1 q''_1+p_1+1)$. Similarly, \eqref{vdu} implies that
\begin{equation}
\label{vdu2}
\| v^i d(u^j)\|'_{q,p}
=\sum_{\ell=0}^{j-1} q_i p_\ell p_{j-\ell-1}
=q_i (p*p)_{j-1} \le C q_i p'_j
= C \| v^i u^j\|_{q,p'}.
\end{equation}
Combining \eqref{dvu2} and \eqref{vdu2}, we see that
\[
\| d(v^i u^j)\|'_{q,p}
=\| d(v^i)u^j+v^i d(u^j)\|'_{q,p}
\le 2C_2 \| v^i u^j\|_{q'',p'}.
\]
Taking into account \eqref{norm_sum_monom}, we conclude that
\[
\| d(a)\|'_{q,p}\le 2C_2 \| a\|_{q'',p'} \qquad (a\in\cT_\alg).
\]
Hence $d$ is continuous for the topologies determined by $P$ and $Q$.
In view of the above remarks, this implies the existence of $D$ and completes the proof.
\end{proof}

\begin{corollary}
The embedding of $\cT_\alg$ into $\cT_\smth$ is a homological epimorphism.
As a consequence, $\cT_\smth$ is quasi-free.
\end{corollary}
\begin{proof}
By Proposition \ref{prop:T_smth_PQ}, we have $\cT_\smth=\cT_{P,P}$, where
$P=\{ p^{(k)}\}_{k\in\N}$ and $p^{(k)}_n=(1+n)^k$ for all $k\in\N$,
$n\in\Z_+$. Given $k\in\N$ and $n\in\Z_+$, we have
\[
(p^{(k)}*p^{(k)})_n
=\sum_{i+j=n} (1+i)^k (1+j)^k
\le (1+n)(1+n)^k (1+n)^k=p^{(2k+1)}_n.
\]
Hence $P*P\prec P$, and Theorem~\ref{thm:main} implies the result.
\end{proof}

\begin{corollary}
The embedding of $\cT_\alg$ into $\cT_\hol$ is a homological epimorphism.
As a consequence, $\cT_\hol$ is quasi-free.
\end{corollary}
\begin{proof}
By Proposition \ref{prop:T_hol_PQ}, we have $\cT_\hol=\cT_{P,P}$, where
$P=\{ p^{(k)}\}_{k\in\N}$ and $p^{(k)}_n=k^n$ for all $k\in\N$,
$n\in\Z_+$. Given $k\in\N$ and $n\in\Z_+$, we have
\[
(p^{(k)}*p^{(k)})_n
=\sum_{i+j=n} k^i k^j = (n+1)k^n \le 2^n k^n
=p^{(2k)}_n.
\]
Hence $P*P\prec P$, and Theorem~\ref{thm:main} implies the result.
\end{proof}


\begin{thebibliography}{99}
\bibitem{Ar}
Aristov, O. Yu.
{\em A global dimension theorem for nonunital and some other separable $C^*$-algebras}.
(Russian) Mat. Sb. 186 (1995), no.~9, 3--18;
translation in Sb. Math. 186 (1995), no.~9, 1223--1239.
\bibitem{BBK}
Ben-Bassat, O. and Kremnizer, K.
{\em Non-Archimedean analytic geometry as relative algebraic geometry}.
Ann. Fac. Sci. Toulouse Math. (6) 26 (2017), no.~1, 49--126.
\bibitem{BZ}
Bondal, A. and Zhdanovskiy, I.
{\em Coherence of relatively quasi-free algebras}.
Eur. J. Math. 1 (2015), no.~4, 695--703.
\bibitem{BFGP}
Bonneau, P., Flato, M., Gerstenhaber, M., and Pinczon, G.
{\em The hidden group structure of quantum groups:
strong duality, rigidity and preferred deformations},
Comm. Math. Phys. \textbf{161} (1994), 125--156.
\bibitem{Cuntz_biv_lok}
Cuntz, J.
{\em Bivariante $K$-Theorie f\"ur lokalkonvexe Algebren und der Chern-Connes-Charakter}.
Doc. Math. 2 (1997), 139--182.
\bibitem{Cuntz_cycl}
Cuntz, J.
{\em Cyclic theory and the bivariant Chern-Connes character}.
Noncommutative geometry, 73--135,
Lecture Notes in Math., 1831, Springer, Berlin, 2004.
\bibitem{CQ1}
Cuntz, J. and Quillen, D.
{\em Algebra extensions and nonsingularity}.
J. Amer. Math. Soc. 8 (1995), no.~2, 251--289.
\bibitem{CQ2}
Cuntz, J. and Quillen, D.
{\em Cyclic homology and nonsingularity}.
J. Amer. Math. Soc. 8 (1995), no.~2, 373--442.
\bibitem{CQ_ex}
Cuntz, J. and Quillen, D.
{\em Excision in bivariant periodic cyclic cohomology}.
Invent. Math. 127 (1997), no.~1, 67--98.
\bibitem{CMR}
Cuntz, J., Meyer, R., and Rosenberg, J. M.
{\em Topological and bivariant $K$-theory}.
Birkh\"auser Verlag, Basel, 2007
\bibitem{Davidson}
Davidson, K. R.
{\em $C^*$-algebras by example}.
American Mathematical Society, Providence, RI, 1996.
\bibitem{Dicks}
Dicks, W.
{\em Mayer-Vietoris presentations over colimits of rings}.
Proc. London Math. Soc. (3) \textbf{34} (1977), no.~3, 557--576.
\bibitem{EschmPut}
Eschmeier, J. and Putinar, M.
{\em Spectral decompositions and analytic sheaves}.
The Clarendon Press, Oxford University Press, New York, 1996.
\bibitem{GL}
Geigle, W. and Lenzing, H.
{\em Perpendicular categories with applications to representations and sheaves}.
J. Algebra \textbf{144} (1991), no.~2, 273--343.
\bibitem{Ginzburg}
Ginzburg, V.
{\em Lectures on noncommutative geometry},
arXiv:math/0506603 [math.AG].
\bibitem{X_dg_ne_1}
Helemskii, A. Ya.
{\em The global dimension of a functional Banach algebra is different from one}.
(Russian) Funkcional. Anal. i Prilo\v{z}en. 6 (1972), no.~2, 95--96;
translation in Funct. Anal. Appl. 6 (1972), no.~2, 166--168.
\bibitem{X1}
Helemskii, A. Ya.
{\em The homology of Banach and topological algebras}.
Mathematics and its Applications (Soviet Series), 41.
Kluwer Academic Publishers Group, Dordrecht, 1989.
\bibitem{X2}
Helemskii, A. Ya. {\em Banach and locally convex algebras}. Oxford Science
Publications. The Clarendon Press, Oxford University Press, New York, 1993.
\bibitem{X_HOA}
Helemskii, A. Ya.
{\em Homology for the algebras of analysis}.
Handbook of algebra, Vol.~2, 151--274, Elsevier/North-Holland, Amsterdam, 2000.
\bibitem{Keller_DCU}
Keller, B.
{\em Derived categories and their uses}.
Handbook of algebra, Vol.~1, 671--701, Elsevier/North-Holland, Amsterdam, 1996.
\bibitem{Khalkhali}
Khalkhali, M.
{\em Basic noncommutative geometry}.
European Mathematical Society (EMS), Z\"urich, 2009.
\bibitem{LeBruyn_book}
Le Bruyn, L.
{\em Noncommutative geometry and Cayley-smooth orders}.
Chapman \& Hall/CRC, Boca Raton, FL, 2008.
\bibitem{Meyer_thesis}
Meyer, R.
{\em Analytic cyclic cohomology},
Ph.D. Thesis, Westf\"alische Wilhelms-Universit\"at M\"unster, 1999,
http://www.arXiv.org/math.KT/9906205.
\bibitem{Meyer_exentire}
Meyer, R.
{\em Excision in entire cyclic cohomology}.
J. Eur. Math. Soc. (JEMS) 3 (2001), no.~3, 269--286.
\bibitem{Meyer_emb}
Meyer, R. {\em Embeddings of derived categories of bornological modules},
arXiv.org:math.FA/0410596.
\bibitem{Meyer_locan}
Meyer, R.
{\em Local and analytic cyclic homology}.
European Mathematical Society (EMS), Z\"urich, 2007.
\bibitem{Murphy}
Murphy, G. J.
{\em $C^*$-algebras and operator theory}.
Academic Press, Inc., Boston, MA, 1990.
\bibitem{NR}
Neeman, A. and Ranicki, A. {\em Noncommutative localization and chain
complexes. I. Algebraic $K$- and $L$-theory}, arXiv.org:math.RA/0109118.
\bibitem{Panarin}
Panarin, K.
{\em The Pimsner-Voiculescu exact sequence for the $K$-theory of holomorphic crossed products}.
Master thesis, HSE University, Faculty of Mathematics, 2017.
\bibitem{Phillips_K_Fre}
Phillips, N. C.
{\em $K$-theory for Fr\'echet algebras}.
Internat. J. Math. 2 (1991), no.~1, 77--129.
\bibitem{Pietsch2}
Pietsch, A.
{\em Zur Theorie der topologischen Tensorprodukte},
Math. Nachr. \textbf{25} (1963), 19--31.
\bibitem{Pietsch}
Pietsch, A.
{\em Nuclear locally convex spaces}.
Springer-Verlag, New York-Heidelberg, 1972.
\bibitem{Pir_qfree}
Pirkovskii, A. Yu.
{\em Arens-Michael envelopes, homological epimorphisms, and relatively
quasi-free algebras}. (Russian). Tr. Mosk. Mat. Obs. \textbf{69} (2008), 34--125;
translation in Trans. Moscow Math. Soc. 2008, 27--104.
\bibitem{Pir_wdgnucl}
Pirkovskii, A. Yu.
{\em Homological dimensions and Van den Bergh isomorphisms for nuclear Fr\'echet algebras}.
(Russian) Izv. Ross. Akad. Nauk Ser. Mat. 76 (2012), no.~4, 65--124;
translation in Izv. Math. 76 (2012), no.~4, 702--759.
\bibitem{Quillen}
Quillen, D.
{\em Higher algebraic $K$-theory. I}. Algebraic K-theory, I: Higher K-theories
(Proc. Conf., Battelle Memorial Inst., Seattle, Wash., 1972), pp. 85--147.
Lecture Notes in Math., Vol. 341, Springer, Berlin, 1973.
\bibitem{Quillen_book}
Quillen, D. and Blower, G.
{\em Topics in cyclic theory}.
Cambridge University Press, 2020.
\bibitem{Schelter}
Schelter, W.
{\em Smooth algebras}.
J. Algebra 103 (1986), no.~2, 677--685.
\bibitem{Stnstrm}
Stenstr\"om, B.
{\em Rings of quotients}.
Springer-Verlag, New York-Heidelberg, 1975.
\bibitem{T1}
Taylor, J. L.
{\itshape Homology and cohomology for topological algebras},
Adv. Math. \textbf{9} (1972), 137--182.
\bibitem{T2}
Taylor, J. L.
{\itshape A general framework for a multi-operator functional calculus},
Adv. Math. \textbf{9} (1972), 183--252.
\bibitem{Voigt_per}
Voigt, C.
{\em Equivariant periodic cyclic homology}.
J. Inst. Math. Jussieu 6 (2007), no.~4, 689--763.
\bibitem{Voigt_loc}
Voigt, C.
{\em Equivariant local cyclic homology and the equivariant Chern-Connes character}.
Doc. Math. 12 (2007), 313--359.
\bibitem{Weibel}
Weibel, C. A.
{\em An introduction to homological algebra}.
Cambridge University Press, Cambridge, 1994.
\bibitem{Witherspn}
Witherspoon, S. J.
{\em Hochschild cohomology for algebras}.
American Mathematical Society, Providence, RI, 2019.
\end{thebibliography}
\end{document}